\documentclass{amsart}

\usepackage{amssymb}
\usepackage{color}
\usepackage{amsxtra} 
\usepackage{mathrsfs} 
\usepackage{yfonts}

\newtheorem{theorem}{Theorem}

\newtheorem{lem}[theorem]{Lemma}
\newtheorem{prop}[theorem]{Proposition}
 
\newtheorem*{theorem*}{Main Theorem} 
\newtheorem*{theorem**}{Theorem}

\theoremstyle{definition}
\newtheorem{definition}[theorem]{Definition}

\theoremstyle{remark}
\newtheorem{remark}[theorem]{Remark} 
\newtheorem*{remark*}{Remark}

\numberwithin{equation}{section}
\numberwithin{theorem}{section} 

\begin{document}

\title[Mathematical analysis of pulsatile flow and vortex breakdown 
]
{
Mathematical analysis of pulsatile flow, vortex breakdown
and  instantaneous  blow-up 
 for the
axisymmetric
 Euler equations
}


\author{Tsuyoshi Yoneda}
\address{Graduate School of Mathematical Sciences, University of Tokyo, Komaba 3-8-1 Meguro, Tokyo 153-8914, Japan} 
\email{yoneda@ms.u-tokyo.ac.jp}


\date{\today} 



\begin{abstract} 
The dynamics along the particle trajectories for the 3D axisymmetric 
 Euler equations
 are considered. 
It is shown that 
if the inflow is rapidly increasing (pushy) in time, the corresponding  
 laminar profile of the incompressible Euler flow is not (in some sense) stable
provided that 
the swirling component  is not zero.
It is also shown that  
 if the vorticity on the axis is not zero (with some extra assumptions), then there is no 
steady flow.
We can rephrase these instability to an instantaneous blow-up.
In the proof, Frenet-Serret formulas and orthonormal moving frame are essentially used.
\end{abstract} 

\maketitle

\section{Introduction} 
\label{sec:Intro} 
We study the dynamics along  the particle trajectories for the 3D axisymmetric 
 Euler equations.
Such  Lagrangian dynamics of the 3D axisymmetric Euler flow (inviscid flow) have already been studied in mathematics (see \cite{C0, C1,C2}). For example, in \cite{C1}, Chae considered a blow-up problem for the axisymmetric 3D incompressible Euler equations with swirl. 
More precisely, he showed that under some assumption of local minima for the pressure on the axis of symmetry
with respect to the radial variations along some particle trajectory, the solution blows up in finite time.

Although the blowup problem of the 3D incompressible Euler equations (also the Navier-Stokes equations) is still an outstanding open problem, in this paper, we focus on a different problem in physics, 
especially, 
``pulsatile flow" and ``vortex breakdown".
In the pulsatile flow study field, Womersley number is  the key. The Wormersley number
comes from oscillating (in time) solutions to the incompressible Navier-Stokes equations in a tube.
Let us explain more precisely.
We define a pipe $\Omega_{\mathcal R}$ as  $\Omega_{\mathcal R}:=\{x\in\mathbb{R}^3: \sqrt{x_1^2+x_2^2}<\mathcal R,\ 0<x_3<\ell
\}$
with its side-boundary  $\partial\Omega_{\mathcal R}=\{x\in\mathbb{R}^3: \sqrt{x_1^2+x_2^2}=\mathcal R,\ 0<x_3<\ell
\}$, top and bottom boundaries:
$\partial\Omega_\mathcal{R}^{top}:=\{x\in\mathbb{R}^3: 0\leq \sqrt{x_1^2+x_2^2}<\mathcal R,\ x_3=\ell
\}$ and $\partial\Omega_\mathcal{R}^{bottom}:=\{x\in\mathbb{R}^3: 0\leq \sqrt{x_1^2+x_2^2}<\mathcal R,\ x_3=0
\}$.
The incompressible Navier-Stokes equations are described as follows:
\begin{equation}
\label{NS eq.}
\partial_tu+(u\cdot \nabla)u-\nu\Delta u=-\nabla p,\quad\nabla \cdot u=0\quad \text{in}\quad \Omega_\mathcal{R},
\quad u=0 \quad \text{on}\quad \partial\Omega_{\mathcal R}
\end{equation}
with $u=u(x,t)=(u_1(x_1,x_2,x_3,t),u_2(x_1,x_2,x_3,t),u_3(x_1,x_2,x_3,t))$ and 
 $p=p(x,t)$.

To give the Womersley number, we need to focus on the axisymmetric Navier-Stokes flow without swirl (see \cite{W}).
If $p_1$ and $p_2$ are the pressure at the ends of the pipe $\Omega_{\mathcal R}$, namely, $\partial\Omega_\mathcal{R}^{top}$
and $\partial\Omega_{\mathcal{R}}^{bottom}$,
the 
pressure gradient can be expressed as $(p_1-p_2)/\ell$ (for the study of pressure boundary conditions on $\partial\Omega_{\mathcal R}^{top}$ and $\partial\Omega_{\mathcal R}^{bottom}$, see \cite{HRT} for example).
If the pressure gradient is time-independent, $(p_1-p_2)/\ell=:p_s$, 
then we can find a stationary Navier-Stokes flow (Poiseuille flow):
\begin{equation}\label{p_s}
u_s(r)=(u_1,u_2, u_3)=(0, 0, \frac{p_s}{4\nu \ell}(\mathcal R^2-r^2)),
\end{equation}
where $r=\sqrt{x_1^2+x_2^2}$.
Note that $u_s$ is also a solution to the  linearized Navier-Stokes equations.
Next we  consider the oscillating pressure gradient case,
\begin{equation}\label{n}
\frac{p_1(t)-p_2(t)}{\ell}=p_oe^{i N t}
\end{equation}
which is periodic in the time.
Then its corresponding solution $u_o=u_o(r,t)$  can be written explicitly by using a Bessel function (see \cite[(8)]{W} and \cite[(1)]{TKWP}) with $u_1=u_2=0$. Thus $u_o$ is also a solution to the linearized Navier-Stokes equations.
Note that $u_o+u_s$ is a time-periodic solution to the Navier-Stokes equations. 
In this study field, the following Womersley number $\alpha$ is the key: 
\begin{equation*}
\alpha=\mathcal R\sqrt{\frac{N}{\nu}}.
\end{equation*}
In \cite{TKWP}, they also defined the oscillatory Reynolds number and 
the mean Reynolds number by using $u_o$ and $u_s$ respectively,
and they investigated how the transition of pulsatile flow $u_o+u_s$ from the laminar to the 
turbulent (critical Reynolds number) is affected by the Womersley number and the oscillatory Reynolds number.
According to their experiment,
measurement at different Womersley numbers yield similar transition behavior,
and variation of the oscillatory Reynolds number also appear to have little effect.
Thus they conclude that the transition seems to be determined only by the mean Reynolds number.
However it seems they did not investigate  the effect of  the  swirl component (azimuthal component),
 and 
our aim here is to 
 show that the non-zero swirl component induces an instability of the laminar profile which is, at a glance, nothing to do with wall turbulence.

On the other hand, in the study of vortex breakdown,
determining the possible flow topologies of the steady axisymmetric Navier-Stokes flow in a cylindrical container (such as $\Omega_\mathcal{R}$)
with rotating end-covers (on $\partial\Omega^{top}_\mathcal{R}$ and $\partial\Omega^{bottom}_\mathcal{R}$)
has been the main subject (see  \cite{BVS,E,H,S} for example, see also \cite{HNY}).
The flow structures and the stability of the flow turns out to be sensitive to changes in the rotation ratio of the two covers.
Using a combination of bifurcation theory for two-dimensional dynamical systems and numerical computations,
Brons-Voigt and Sorensen \cite{BVS} systematically determined the possible flow topologies of the steady vortex breakdown 
in the axisymmetric flow. Their basic idea is to analyze the streamlines of the ordinary differential equations 
(c.f. the definition of axis-length streamline \eqref{axis-length streamline} and axis-length trajectory: Definition \ref{axis-length trajectory} in this paper).
For the detail, see Figure 1 and Section 3 in \cite{BVS}.
Our aim here is to show that non-zero swirl component with laminar profile on the axis (with some extra assumptions)
creates unsteady flow.

\begin{remark}
These mathematical analysis must be applicable to a study of reduced cardiovascular 1D model \cite[Section 10]{FQV}. 
If the blood flow is in large and medium sized vessels, the flow is governed by the usual incompressible Navier-Stokes equations. 
To obtain the reduced model from the Navier-Stokes equations, we need to assume the flow is always unilateral laminar flow,
especially,
 the axis direction of the flow $u_3$ is assumed to satisfy
\begin{equation}\label{unilateral flow condition}
\int_{\Omega_\mathcal{R}}u_3(x_1,x_2,x_3,t)^2dx_1dx_2=a \left(\int_{\Omega_\mathcal{R}}u_3(x_1,x_2,x_3,t)dx_1dx_2\right)^2
\end{equation}
for some positive constant $a>0$ (see \cite[(10.18)]{FQV}).
However, in this setting, it is unclear whether or not such  condition \eqref{unilateral flow condition} is always valid. For example, if the flow is not unilateral, containing the  reverse flow (possibly, turbulence), 
then $a$ may become infinity.
\end{remark}

Since we  would not like to  take the boundary layer into account (instead, we focus on behavior of the interior flow),  it is still valid to consider  
a simpler model: the inviscid flow in $\Omega_\mathcal{R}$.
The incompressible Euler equations (inviscid flow) are expressed as follows:
\begin{eqnarray}
\label{Euler eq.}
& &D_tu:=\partial_tu+(u\cdot \nabla)u=-\nabla p,\quad\nabla \cdot u=0\quad \text{in}
\quad \Omega_\mathcal{R},\\
\nonumber
& & \quad u|_{t=0}=u_0,
\quad u\cdot n=0 \quad \text{on}\quad \partial\Omega_\mathcal{R},\\
& &
u(x,t)|_{x_3= 0}= (0,0,U_{in}(r,t))\quad\text{with}\quad U_{in}(r,t)>0,
\end{eqnarray}
where $r=\sqrt{x_1^2+x_2^2}$ and $n$ is a unit normal vector on $\partial\Omega_\mathcal{R}$.
Note that the boundary condition here  is not important anymore.

Notations ``$\approx$" and ``$\lesssim$" are convenient. The notation ``$a \approx b$" means there is a positive constant $C>0$ such that 
\begin{equation*}
C^{-1}a\leq b\leq Ca,
\end{equation*}
and ``$a\lesssim b$" means that there is a positive constant $C>0$ such that 
\begin{equation*}
 a\leq Cb.
\end{equation*}
In the pulsatile flow case, we consider the following inflow setting:
\begin{itemize}

\item
$U_{in}=U_s(r)+U_o(r)g(t)$
with rapidly increasing $g$ (in time)  and 
\begin{equation*}
|U_o(r,t)|\approx 1,\quad
\sup_{1\leq j+k\leq 2}\left(|\partial_r^j\partial_t^kU_{o}(r,t)|+|\partial_r^j\partial_t^kU_{s}(r,t)|\right)\lesssim1,
\end{equation*}
\end{itemize}

Throughout this paper we always assume existence of smooth solutions in $\Omega_\mathcal {R}\times[0,\infty)$ (we can regard nonuniqueness, nonexistence and blowup as some kind of `` strong instability"). 

\begin{remark}
According to the boundary layer theory,
outside the boundary layer the fluid motion is accurately described by the Euler flow.
Thus the above  simplification seems (more or less) valid.
For the recent progress on the mathematical analysis of the boundary layer, 
see \cite{MM}.
\end{remark}

\section{Geometry setting and the main results}

To describe the main theorems, we need to give a geometry setting.
First we define the particle trajectory.

\begin{definition} (Particle trajectory $\Phi_*$.)\ 
For given time-dependent smooth vector field $u=u(x,t)$,
the associated Lagrangian flow $\Phi_*(t)$
is a solution of the following initial value problem 
\begin{align} \label{eq:flowC} 
&\frac{d}{dt}\Phi_*(x,t) = u(\Phi_*(x,t),t), 
\\  \label{eq:flow-icC} 
&\Phi_*(x,0) = x\in
\Omega_\mathcal{R}. 
\end{align}
\end{definition}
Throughout this paper we always assume the vector field $u$ is unilateral, that is, $u_3>0$.
Also define the axis-length streamline $\tilde \Phi(z)$.
\begin{definition} (Axis-length streamline $\tilde \Phi$.)\ 
for fixed $t>0$, let $\tilde \Phi$ be such that 
\begin{equation}\label{axis-length streamline}
\partial_z\tilde \Phi(z)=\frac{u(x,t)}{u_3(x,t)}\bigg|_{x=\tilde \Phi(z)}\quad
\end{equation}
with the initial point $\tilde \Phi(0)=(x_1,x_2,0).$
\end{definition}
Later, we use the following axis-length trajectory $\Phi$.
\begin{definition}\label{axis-length trajectory} (Axis-length trajectory $\Phi$.)
Let $Z_*(t):=\Phi_*(t)\cdot e_z$ (with $e_z=(0,0,1)$) and since the flow is unilateral, we can define its inverse $Z^{-1}_{_*t}(z)=t$. 
In this case we see $\partial_zZ^{-1}_{*t}=1/\partial_tZ_*=1/u_3$.
Let $\Phi$ be such that 
 $\Phi(z)=\Phi_*(x,Z^{-1}_{*t}(z))$.
\end{definition}
We restrict our vector field to the axi-symmetric one.
Let $e_r:= x_h/|x_h|$,
$e_\theta:=x_h^\perp/|x_h|$ and
  $e_z=(0,0,1)$ with
 $x_h=(x_1,x_2,0)$, $x_h^\perp=(-x_2,x_1,0)$.
The vector valued function $u$ can be rewritten as   $u=v_re_r+v_\theta e_\theta+v_ze_z$,  
where $v_r=v_r(r,z,t)$, $v_\theta=v_\theta(r,z,t)$, $v_z=v_z(r,z,t)$, $v_{r,0}=v_r(r,z,0)$, $v_{\theta,0}=v_\theta(r,z,0)$ and  $v_{z,0}=v_z(r,z,0)$
with $r=|x_h|$ and $z=x_3$.

We  define a Lagrangian flow on the meridian plane ($r$-$z$ plane).

\begin{definition}(Lagrangian flow on the meridian plane.)\ 
 Let  $Z_*$ and $R_*$ be such that 
\begin{eqnarray}\label{2D-trajectory-1}
& &\frac{d}{dt}Z_*(t)=v_z(R_*(t),Z_*(t),t),\\
\nonumber
& &Z_*(0)=z_0
\end{eqnarray}
and
\begin{eqnarray}\label{2D-trajectory-2}
& &\frac{d}{dt}R_*(t)=v_r(R_*(t),Z_*(t),t),\\
\nonumber
& &R_*(0)=r_0
\end{eqnarray}
with $Z_*(t)=Z_*(r_0,z_0,t)$ and $R_*(t)=R_*(r_0,z_0,t)$.
\end{definition}
Note $Z_*$ is already defined in Definition \ref{axis-length trajectory}. 

\begin{remark}
We can rephrase $Z_*$ and $R_*$ by using the stream function (see (2.2) in \cite{BVS} for example).
\end{remark}

\begin{remark}
The   axisymmetric Euler equations can be expressed  
as follows: 
\begin{eqnarray}
\partial_t v_r+v_r\partial_rv_r+v_z\partial_zv_r-\frac{v_\theta^2}{r}+\partial_r p&=&0
,\\
\label{axisymmetricEuler-1}
\partial_tv_\theta+v_r\partial_rv_\theta+v_z\partial_zv_\theta+\frac{v_rv_\theta}{r}&=&0
,\\
\partial_tv_z+v_r\partial_rv_z+v_z\partial_zv_z+\partial_zp&=&0
,\\
\label{axisymmetricEuler-2}
\frac{\partial_r(rv_r)}{r}+\partial_zv_z&=&0.
\end{eqnarray}
In this paper we  use \eqref{axisymmetricEuler-1} which is independent of the pressure term.
\end{remark}

\vspace{0.3cm}

\begin{remark}\label{Stream-shell near the boundary} 
(Axisymmetric axis-length  streamline.)\
For fixed $t>0$, 
 $\tilde \Phi(z)$ can be explicitly expressed as  
\begin{equation*}
\tilde \Phi(\tilde r_0,z,t)=\tilde\Phi (z):=(\tilde R(z)\cos \tilde \Theta(z), \tilde  R(z)\sin \tilde \Theta(z), z)
\end{equation*} 
with
$\tilde R(z)=\tilde R(\tilde r_0,z,t)$, $\tilde R(\tilde r_0,0,t)=\tilde r_0$, $\tilde \Theta(z)=\tilde \Theta(z,t)$.
We easily see 
\begin{equation*}
\partial_z\tilde \Phi\cdot e_z=1,\quad \partial_z\tilde\Phi\cdot e_r=\partial_z\tilde R=\frac{v_r}{v_z}\quad\text{and}\quad
\partial_z \tilde\Phi\cdot e_\theta=\tilde R\partial_z\tilde\Theta=\frac{v_\theta}{v_z}.
\end{equation*}
Since $\partial_{\tilde r_0}\tilde R>0$ by the smoothness,
we have 
its inverse $\tilde r_0=\tilde R^{-1}(r,z,t)$.
\end{remark}
\begin{remark}\label{axis-length trajectory} (Axisymmetric axis-length trajectory.)
Also $\Phi$ can be explicitly expressed as 
\begin{equation*}
\Phi(z)=(R(z)\cos\Theta(z),R(z)\sin\Theta(z),z)
\end{equation*}
with
$R(z)=R(r_0,\theta_0,z_0,z)$, $\Theta(z)=\Theta(r_0,\theta_0,z_0,z)$,
$R(r_0,\theta_0,z_0,z_0)=r_0$ and $\Theta(r_0,\theta_0,z_0,z_0)=\theta_0$.
Note that $\tilde R(t)|_{t=Z^{-1}_{*t}(z)}=R(z)$.
\end{remark}

In order to show that  the non-zero swirl component induces the instability,
we need to measure appropriately the rate of disturbing laminar profile of the  flow.  
We now give the key definition.
\begin{definition} (Rate of disturbing laminar profile.)\ 
 We 
define ``rate of disturbing laminar profile" $L^0$,  $L^x$ and $L^t$ as follows: for $(\tilde r_0,z)\in [0,\mathcal R)\times (0,\ell)$,
\begin{eqnarray*}
L^0(\tilde r_0,z,t)&=&|\partial_{\tilde r_0}\tilde R(\tilde r_0,z,t)|+|(\partial_r\tilde R^{-1})(\tilde R(\tilde r_0,z,t),z,t)|\\,
L^x(\tilde r_0,z,t)
&:=&
\sum_{\stackrel{1\leq j+k\leq 3,}{(j,k)\not=(0,1)}}
|\partial_z^j\partial_{\tilde r_0}^k\tilde R(\tilde r_0,z,t)|
+
 \sum_{\stackrel{1\leq j+k\leq 3}{(j,k)\not=(0,1)}}|( \partial^j_z\partial^k_r\tilde R^{-1})(\tilde R(\tilde r_0,z,t),z,t)|\\
L^t(\tilde r_0,z,t)
&=&
\sum_{\stackrel{1\leq i\leq 2}{2\leq i+j+k\leq 3}}|\partial_t^i\partial_z^j\partial_{\bar r_0}^k\tilde R(\tilde r_0,z,t)|
+
\sum_{\stackrel{1\leq i\leq 2}{1\leq i+j+k\leq 2}}|(\partial_t^i\partial_z^j\partial_{r}^k\tilde R^{-1})(\tilde R(\tilde r_0,z,t),z,t)|
\end{eqnarray*}
Note that $L^0$ and $L^x$ do not include any time derivative, while, $L^t$ includes it.
We can see that if $L^t$ is not zero, then the flow cannot be any steady flow.



\begin{remark}
Minumum value of $L^0$ is $2$, since $|\partial_r\tilde R^{-1}|=1/|\partial_{\bar r_0}\tilde R|$.
\end{remark}

\end{definition}

\begin{remark}\label{no swirl}
The typical Euler flow $u(x,t)=(0,0, g(t))$, namely,
a bunch of stationary straight tubes $\tilde R(\tilde r_0,z,t)\equiv \tilde r_0$ 
is the typical laminar flow.
In this case 
\begin{equation*}
L^0\equiv 2,\quad
L^x\equiv 0
\quad\text{and}
\quad
L^t\equiv 0
\end{equation*}
for any $g$.
\end{remark}
\begin{remark}
Streamlines of outside bubbles which are attaching on the axis (see $B_1$, $B_2$, $B_3$, $C_1$, $C_2$, $D$, $E$, $F$ in Figure 1 in \cite{BVS}) may create large $L^x$ and/or $L^0$. Moreover,  at a hyperbolic saddle (or stagnation point),
they may be infinity. 
\end{remark}
 
Now we give the main theorems.

\begin{theorem} (Pulsatile flow case.)
Let $\tilde r_0(t)$ and $z(t)$ be another expression of  particle trajectory such that 
\begin{equation*}
\tilde \Phi(\tilde r_0(t),z(t),t)=\Phi_*(x,t)
\end{equation*}
and let $D_\gamma$ be a non-zero swirl region such that $D_\gamma:=\{x:|u_0(x)\cdot e_\theta|>\gamma\}$. 
Assume $D_\gamma\not=\emptyset$ for the corresponding initial data, and assume there is a unique  smooth solution to the Euler equations \eqref{Euler eq.} in $t\in[0,T)$.
Then there is a smooth function $g$, $f$ and discrete-time $\{t_j\}_j$ such that 
\begin{equation*}
|g|\approx 1,\quad g'(t_j)\to\infty,\quad g''(t_j)\to\infty\quad\text{and}\quad f(t_j)\to\infty \quad (t_j\to T\quad\text{as}\quad j\to\infty),
\end{equation*}
and the following case must occur:
\begin{equation*}
L^t(\tilde r_0(t_j),z(t_j),t_j)\geq f(t_j)\quad\text{for}\quad x\in D_\gamma\quad\text{and}\quad j=1,2,\cdots.
\end{equation*}
\end{theorem}

\begin{theorem}
(Vortex breakdown case.)\ 
Assume there is a unique  smooth solution to the Euler equations \eqref{Euler eq.}.
For any  $\epsilon_1>0$ and $\epsilon_2>0$ ($\epsilon_1\ll\epsilon_2$), there is $\delta>0$ such that if

\begin{eqnarray*}
|v_z(0,z,0)|&\approx& 1\\
|\partial_rv_\theta(0,z,0)|&\approx& 1/\epsilon_1,\\
\sum_{\stackrel{0\leq j+k\leq 3}{(j,k)\not=(1,0)}}|\partial_r^j\partial_z^kv_\theta(0,z,0)|&\lesssim& 1/\delta,\\
\sum_{0\leq j+k\leq 2}|\partial_r^j\partial_z^kv_r(0,z,0)|&\lesssim& 1/\delta,\\
\sum_{1\leq j+k\leq 2}|\partial_r^j\partial_z^kv_z(0,z,0)|&\lesssim& 1/\delta
\end{eqnarray*}

for some $z$,
 then there is no stationary Euler flow near the initial time, that is,
\begin{equation*}
|\partial_tv_z(0,z,0)|+
|\partial_t\partial_rv_r(0,z,0)|+|\partial_t\partial_rv_z(0,z,0)|+
|\partial_t\partial_zv_z(0,z,0)|+|\partial_t\partial_r^2v_z(0,z,0)|
>1/\epsilon_2.
\end{equation*}
Note that $\delta$ becomes smaller if $\epsilon_1$ becomes smaller, to the contrary, 
$\delta$ becomes larger if $\epsilon_2$ becomes smaller. 
\end{theorem}

\begin{remark}
Roughly saying,  $\partial_rv_{\theta,0}$ on the 
axis should be corresponding to  rotating top and bottom boundaries: $\Omega_1$ in \cite[(2.1)]{BVS}. 
\end{remark}



\section{Explicit formulas of $v_r$, $v_z$, $v_\theta$ and $D_t|u|$.}\label{explicit formulas}

Before we prove the main theorems, in this section, we give explicit formulas of $v_r$, $v_z$, $v_\theta$ and $\partial_t|u(\Phi(x,t),t)|$ by using $\tilde R$ and $U_{in}$.
First we construct $v_z$ and $v_r$. To do so, we define the cross section of the stream-tube (annulus). Let 
$B_{0}(\tilde r_0)=\{x\in\mathbb{R}^3: |x_h|<\bar r_0,\ x_3=0\}$ and let
\begin{equation*}
  A(\tilde r_0,z,\epsilon,t):=\bigcup_{r'\in B_{0}(\tilde r_0+\epsilon)\setminus B_{0}(\tilde r_0)}
\tilde \Phi(r',z,t).
\end{equation*}
We see that  its measure is  
\begin{equation*}
|A(\tilde r_0,z,\epsilon,t)|=\pi\left(\tilde R(\tilde r_0,\epsilon,z,t)^2-\tilde R(\tilde r_0,z,t)^2\right).
\end{equation*}
\begin{definition} (Inflow propagation.)\ 
Let $\rho$ be such that  

\begin{equation*}
\rho(\tilde r_0,z,t):=\lim_{\epsilon\to 0}\frac{|A(\tilde r_0,0,\epsilon,t)|}{|A(\tilde r_0,z,\epsilon,t)|}.
\end{equation*}
\end{definition}
We see that
\begin{equation*}
\rho(\tilde r_0,z,t)=\frac{\partial_{\tilde r_0}\tilde R(\tilde r_0,0,t) \tilde R(\tilde r_0,0,t)}{\partial_{\tilde r_0}\tilde R(\tilde r_0,z,t) \tilde R(\tilde r_0,z,t)}=\frac{\tilde r_0}{\partial_{\tilde r_0}\tilde R(\tilde r_0,z,t) \tilde R(\tilde r_0,z,t)}
=\frac{2\tilde r_0}{\partial_{\tilde r_0}\tilde R(\tilde r_0,z,t)^2}.
\end{equation*}

\begin{remark}
Since $\tilde R(0,z,t)\equiv 0$, we see that
\begin{equation*}
\lim_{\tilde r_0\to 0}\rho=\frac{1}{(\partial_{\tilde r_0}\tilde R)^2}
\quad\text{and}\quad 
\lim_{\tilde r_0\to 0}\partial_{\tilde r_0}\rho=-\frac{2\partial_{\tilde r_0}^2\tilde R}{(\partial_{\tilde r_0}\tilde R)^3}
\end{equation*}
on the axis.
\end{remark}

\vspace{0.5cm}

Since 
\begin{equation*}
2\pi \int_{\tilde R(\tilde r_0,z,t)}^{\tilde R(\tilde r_0+\epsilon,z,t)}u_z(r',z,t)r'dr'=2\pi\int_{\tilde r_0}^{\tilde r_0+\epsilon}u_z(r',0,t)r'dr'
\end{equation*}
 by divergence-free and Gauss' divergence theorem,
 we can figure out $v_z$ by using the inflow propagation $\rho$, 
\begin{eqnarray*}
v_z(r,z,t)&=&\lim_{\epsilon\to 0}\frac{2\pi}{|A(\tilde r_0,z,\epsilon,t)|}\int_{\tilde R(\tilde r_0,z,t)}^{\tilde R(\tilde r_0+\epsilon,z,t)}v_z(r',z,t)r'dr'\\
&=&
\lim_{\epsilon\to 0}\frac{|A(\tilde r_0,0,\epsilon,t)|}{|A(\tilde r_0,z,\epsilon,t)|}\frac{2\pi}{|A(\tilde r_0,0,\epsilon,t)|}\int_{\tilde r_0}^{\tilde r_0+\epsilon}v_z(r',0,t)r'dr'\\
&=&
\rho(\tilde r_0,z,t)u_z(\tilde r_0,0,t).
\end{eqnarray*}
Thus we have the following proposition.
\begin{prop}\label{formula of velocities}
We have the following formulas of $v_z$ and $v_r$:
\begin{equation}\label{!}
v_z(r,z,t)
=
\nonumber
\rho(\tilde R^{-1}(r,z,t),z,t)U_{in}(\tilde R^{-1}(r,z,t),t)
\end{equation}
and 
\begin{equation}\label{!!}
v_r(r,z,t)=(\partial_z\tilde R)(\tilde R^{-1}(r,z,t), z,t)v_z(r,z,t).
\end{equation}
\end{prop}

\begin{remark}\label{remark on explicit formulas}
Recall that $e_\theta=(-\sin\Theta(z),\cos\Theta(z),0)$ and $e_r=(\cos\Theta(z),\sin\Theta(z),0)$.
We also have the following explicit formulas of $\Theta'$ and $R'$ ($\Theta$ and $R$ already appeared in the axis-length trajectory. See Remark \ref{axis-length trajectory}):
\begin{eqnarray*}
\partial_z\Phi\cdot e_\theta&=&\frac{\partial_t\Phi_*\cdot e_\theta}
{v_z}=R(z)\Theta'(z)=\frac{v_\theta(R(z),z,Z^{-1}_{*t}(z))}{v_z(R(z),z,Z^{-1}_{*t}(z))},\\
\partial_z\Phi\cdot e_r&=&
\frac{\partial_t\Phi_*\cdot e_r}{v_z}=
\frac{v_r(R(z),z,Z^{-1}_{*t}(z))}{v_z(R(z),z,Z^{-1}_{*t}(z))}\\
&=&
\left(\partial_z\tilde R\right)\left(\tilde R^{-1}(R(z),z,Z^{-1}_{*t}(z)),z, Z^{-1}_{*t}(z)\right)=R'(z).\\
\end{eqnarray*}
Moreover, along the axis,
\begin{equation*}
\lim_{r_0\to 0}\Theta'(z)=\frac{(\partial_rv_\theta)(0,z,Z^{-1}_{*t}(z))}{v_z(0,z,Z^{-1}_{*t}(z))}.
\end{equation*}

\end{remark}
\begin{remark}

For the vortex breakdown case, we have the following estimates on 
$\Theta'$, $\Theta''$, $\Theta'''$, $R'$, $R''$ and $R'''$:
\begin{equation*}
|\Theta'|\approx1/\epsilon_1,\quad |\Theta''|\lesssim1/\delta\quad\text{and}
\quad |\Theta'''|\lesssim 1/\delta.
\end{equation*}
Let $r=R(r_0,z_0,t)$.
Moreover we have that 
\begin{eqnarray*}
R'&=&\frac{v_r}{v_z}=C(\delta)r+O(r^2),\\
R''&=&
\frac{\partial_rv_r}{v_z}R'+\frac{\partial_zv_r}{v_z}+\frac{\partial_tv_r}{v_z\partial_tZ}-\frac{v_r}{v_z^2}\partial_rv_zR'
-\frac{v_r}{v_z^2}\partial_zv_z-\frac{v_r}{v_z^2}\frac{\partial_tv_z}{\partial_tZ}\\
&=&
C(\delta)r+O(r^2),\\
R'''&=& C(\delta)r+O(r^2),
\end{eqnarray*}
where $C(\delta)$ is a  positive constant depending only on $\delta$ (if $\delta\to 0$, then $C(\delta)\to 0$).

\end{remark}

\vspace{1cm}

Next we construct $v_\theta$.
By \eqref{axisymmetricEuler-1} we see that 
\begin{eqnarray*}
\partial_t v_\theta(R_*(t),Z_*(t),t)&=&
-\frac{v_r(R_*(t),Z_*(t),t)v_\theta(R_*(t),Z_*(t),t)}{R_*(t)}.
\end{eqnarray*}
Applying the Gronwall equality,  we see
\begin{eqnarray*}
\nonumber
v_\theta(R_*(t),Z_*(t),t)&=&
v_\theta(r_0,z_0,0)\exp\bigg\{-\int_0^{t}
\frac{v_r(R_*(t'),Z_*(t'),t')}{R_*(t')}dt'\bigg\}\\
\end{eqnarray*}
and then 
\begin{eqnarray*}
\nonumber
v_\theta(R(z),z,Z^{-1}_{*t}(z))&=&
v_\theta(r_0,z_0,0)\exp\bigg\{-\int_0^{Z^{-1}_{*t}(z)}
\frac{v_r(R_*(t'),Z_*(t'),t')}{R_*(t')}dt'\bigg\}\\
\end{eqnarray*}
and
\begin{eqnarray}\label{v_theta formula}
\nonumber
v_\theta(r,z,t)&=&
v_\theta(r_0, z_0,0)\exp\bigg\{-\int_0^{t}
\frac{v_r(R_*(r_0,z_0, t'),Z_*(r_0,z_0,t'),t')}{R_*(r_0,z_0,t')}dt'\bigg\}\\
\end{eqnarray}
with $r_0=R^{-1}_*(r,z,t)$ and  $z_0=Z^{-1}_*(r,z,t)$ (distinguish with $Z^{-1}_{*t}$).
In order to estimate spatial derivatives on $v_\theta$, 
first we consider a non-incompressible 2D-flow composed by $R_*$ and $Z_*$.
Let us denote $\phi_{2D}=\phi_{2D}(t)=(R_*(t),Z_*(t))$, $\phi_{2D}^{-1}=(R^{-1}_*,Z^{-1}_*)$ and $D\phi_{2D}$ be its Lagrangian deformation:
\begin{equation*}
D\phi_{2D}=
\begin{pmatrix}
\partial_{r_0}R_*&\partial_{z_0}R_*\\
\partial_{r_0}Z_*&\partial_{z_0}Z_*
\end{pmatrix}.
\end{equation*}
We see $\det (D\phi_{2D})=\partial_{r_0}R_*\partial_{z_0}Z_*-\partial_{z_0}R_*\partial_{r_0}Z_*$ and thus we have 
\begin{equation*}
D(\phi_{2D}^{-1})=(D\phi_{2D})^{-1}=\frac{1}{\det D\phi_{2D}}
\begin{pmatrix}
\partial_{z_0}Z_*&-\partial_{z_0}R_*\\
-\partial_{r_0}Z_*&\partial_{r_0}R_*
\end{pmatrix}.
\end{equation*}
A direct calculation with \eqref{axisymmetricEuler-2}, \eqref{2D-trajectory-1} and \eqref{2D-trajectory-2} yields
\begin{equation*}
\frac{d}{dt}(\det D\phi_{2D})=(\partial_rv_r+\partial_zv_z)(\det D\phi_{2D})=-\frac{v_r}{R_*(t)}(\det D\phi_{2D}).
\end{equation*}
Thus
\begin{equation*}
\det D\phi_{2D}(t)=\det D\phi_{2D}(0)\exp\bigg\{-\int_0^t\frac{v_r(R_*(\tau),Z_*(\tau),\tau)}{R_*(\tau)}d\tau\bigg\}.
\end{equation*}
Since $|v_r/r|\approx |\partial_rv_r|\lesssim 1$ near the axis,
we have  
\begin{equation*}
\det D\phi_{2D}\approx 1\quad\text{near the initial time}.
\end{equation*}
Since we have already controlled $\det D\phi_{2D}$,
it suffices to estimate $\partial_{r_0}R_*$, $\partial_{r_0}Z_*$, $\partial_{z_0}R_*$ and $\partial_{z_0}Z_*$ respectively. From Proposition \ref{formula of velocities},
We see the following: 
\begin{equation*}
\partial_t\partial_{z_0}Z_*(t)=\bigg[\partial_{z_0}R_*\partial_{\tilde r_0}\rho\partial_r\tilde R^{-1}+
\partial_{z_0}Z_*\partial_{\tilde r_0}\rho\partial_z\tilde R^{-1}+\partial_{z_0}Z_*\partial_z\rho\bigg]U_{in}\\
\end{equation*}
\begin{equation*}
\partial_t\partial_{z_0}R_*(t)=\bigg[\partial_{z_0}\partial_z\tilde R\partial_{r}\tilde R^{-1}\partial_{z_0}R_*+
\partial_{r_0}\partial_z\tilde R\partial_z\tilde R^{-1}\partial_{z_0}Z_*+\partial_z^2\tilde R\partial_{z_0}Z_*\bigg]U_{in}+(v_z\text{ part}).
\end{equation*}
Then we can construct a Gronwall's inequality of $|\partial_{z_0}Z|+|\partial_{z_0}R|$,  
 that is 
\begin{equation*}
|\partial_{z_0}Z_*|+|\partial_{z_0}R_*|\lesssim e^{Ct},
\end{equation*}
where $C$ is depending on $L^0$, $L^x$ and $L^t$.
Again, 
we just take integration in time, we have  
\begin{equation*}
\partial_{z_0}Z_*(t)=1+\int_0^t\partial_{z_0}v_zdt',
\end{equation*}
and this is the  explicit formula of $\partial_{z_0}Z_*$.
In a small time interval,
we have $\partial_{z_0}Z_*\approx 1$ and 
by the same calculation, $\partial_{z_0}R_*\approx 0$, $\partial_{r_0}Z_*\approx 0$ and $\partial_{r_0}R_*\approx 1$.
By the above estimates, we can estimate derivatives on $v_\theta$.

Now we figure out the explicit formula of $\partial_t|u(\Phi_*,t)|$.
Recall that   the  particle trajectory $\Phi_*(x,t)$ satisfies 
\begin{equation*}
\Phi_*(x,t)
=(R_*(t)\cos\Theta_*(t),R_*(t)\sin\Theta_*(t),Z_*(t)).
\end{equation*}
 Then, by $u=v_re_r+v_\theta e_\theta+v_ze_z$ with 
$e_\theta=(-\sin\Theta_*(t),\cos\Theta_*(t),0)$ and 
$e_r=(\cos\Theta_*(t),\sin\Theta_*(t),0)$,  we see that 
\begin{equation}\label{multiply}
\frac{1}{2}\partial_t|u(\Phi_*(x,t),t)|^2=\partial_t u\cdot u
=\partial_t v_rv_r+\partial_tv_\theta v_\theta+\partial_tv_zv_z
\end{equation}
along the trajectory.
In fact, since
\begin{equation*}
\partial_t\Phi_*=(\partial_t R_*\cos\Theta_*,\partial_t R_*\sin\Theta_*, \partial_tZ_*)+\partial_t\Theta_*(-R_*\sin\Theta_*,R_*\cos\Theta_*, 0),
\end{equation*}
and 
\begin{equation*}
v_\theta=\partial_t\Phi_*\cdot e_\theta=(\partial_t\Theta_*) R_*,
\end{equation*}
we see $\partial_t \Theta_*=v_\theta/R_*$.
We multiply $u=v_r e_r+v_\theta e_\theta+v_ze_z$ to
\begin{equation*}
\partial_t u=\partial_tv_r e_r+\partial_tv_\theta e_\theta+\partial_tv_z e_z+
v_r\partial_t\Theta_* e_\theta-v_\theta\partial_t\Theta_* e_r,
\end{equation*}
then we have \eqref{multiply}.
Thus we have the following explicit formula:
\begin{equation*}
D_t|u|=\frac{2D_t|u|^2}{|u|}=\frac{\partial_tv_rv_r+\partial_tv_\theta v_\theta+\partial_tv_z v_z}{\sqrt{v_r^2+v_\theta^2+v_z^2}}.
\end{equation*}
Combining  the  Lagrangian deformation on $R_*$ and $Z_*$,  we also have
the explicit formulas of   
 $\partial_z\partial_t|u(\Phi(x,t),t)|$ and $\partial_r\partial_t|u(\Phi(x,t),t)|$.

\section{Estimates on curvature and torsion along particle trajectory.}\label{estimates on curvature and torsion}

Let us define  the arc-length trajectory $\phi(s):=\Phi(z(s))$ with smooth function $z(s)$ such that $z'(s)=|(\partial_z\Phi)(z(s))|^{-1}$.
We also define the unit tangent vector $\tau$ as 
\begin{equation*}
\tau(s)=\partial_s \phi(s),
\end{equation*}
the unit curvature vector $n$ as $\kappa n=\partial_s \tau$ with a curvature function $\kappa(s)>0$,
the unit torsion vector $b$ 
as : $b(s):=\pm\tau(s)\times n(s)$ ($\times$ is an exterior product)
 with a torsion function to be positive $T(s)>0$ (once we restrict $T$ to be positive, then the direction of $b$ can be uniquely determined).
From $\kappa n$, we can figure out 
the curvature constant $\kappa:=|\partial_s^2\phi|$ and corresponding unit normal vector:
$n=\partial_s^2\phi/|\partial_s^2\phi|$.
Thus, theoretically, we can explicitly figure out $\kappa$ and $\partial_s \kappa$
by using $R$ and $\Theta$.
First, $\tau$ and  $\kappa n$ are expressed as 
\begin{equation*}
\tau=(\partial_z\Phi)z',\quad \kappa n=\partial_s^2\phi=\partial_z^2\Phi(z')^2+\partial_z\Phi z''.
\end{equation*}
Then direct calculations yield
\begin{eqnarray*}
\partial_z\Phi(x, z)&=&
(-R\Theta'\sin\Theta,R\Theta'\cos\Theta,1)+(R'\cos\Theta,R'\sin\Theta, 0)
,\\
\partial_z^2\Phi(x,z)&=&
-R(\Theta')^2(\cos\Theta,\sin\Theta,0)+(-R\Theta''\sin\Theta, R\Theta''\cos\Theta,0)
\\
& &
+
R''(\cos\Theta,\sin\Theta,0)+2R'\Theta'(-\sin\Theta,\cos\Theta,0),\\
z'(s)&=&|\partial_z\Phi|^{-1}=(1+(R')^2+(R\Theta')^2)^{-1/2},\\
z''(s)&=&-(1+(R')^2+(R\Theta')^2)^{-2}(R'R''+R\Theta'(R'\Theta'+R\Theta'')).\\
\end{eqnarray*}
Therefore
\begin{eqnarray*}
\kappa^2&=&|\kappa n|^2
=|\partial_z^2\phi|^2(z')^4+
2(\partial_z\phi\cdot\partial_z^2\phi)(z')^2z''+
|\partial_z\phi|^2(z'')^2\\
&=&
\left[R^2(\Theta')^4-2R(\Theta')^2R''+(R'')^2+(R\Theta'')^2+4R'\Theta'R\Theta''+4(R'\Theta')^2\right]\\
& &
\times (1+(R')^2+(R\Theta')^2)^{-2}\\
&+&
2\left[-R'R(\Theta')^2+R'R''+R^2\Theta'\Theta''-2R(\Theta')^2R'\right]\\
& &
\times (1+(R')^2+(R\Theta')^2)^{-1}(-1)(1+(R')^2+(R\Theta')^2)^{-2}(R'R''+R\Theta'(R\Theta'+R\Theta''))\\
&+&
\left[(R\Theta')^2+(R')^2\right] (1+(R')^2+(R\Theta')^2)^{-4}(R'R''+R\Theta'(R'\Theta'+R\Theta''))^2.
\end{eqnarray*}

From the above explicit formulas of $\kappa$,
 we can figure out the explicit formula of 
$\partial_s\kappa$ (omit its detail) which will be important in the proof of the main theorems.

\begin{remark}
\begin{itemize}

\item (The vortex breakdown case.)\ 
If $\Theta'$ is larger than the other terms, we have
\begin{equation*}
\partial_s\kappa\approx R\Theta'\Theta''
\end{equation*}
which is a controllable term.

\item (Instantaneous blowup case in Appendix.)\ 
If $\Theta''$ is larger than $\Theta'$, and $\Theta'''$ is larger than $\Theta''$, then we have 
\begin{equation*}
\partial_s\kappa\approx R\Theta'''
\end{equation*}
which will be the dominant term.
\end{itemize}
\end{remark}

\section{Rewrite Euler equations by using curvature and torsion}

In this section we rewrite the Euler equations by using curvature and torsion.
The basic idea comes from  Chan-Czubak-Y \cite[Section 2.5]{CCY}, more originally, see Ma-Wang \cite[(3.7)]{MW}.
They considered 2D separation phenomena using elementary differential geometry. 
The key idea here is  ``local pressure estimate" on a normal coordinate in $\bar \theta$, $\bar r$ and $\bar z$ valuables.
Two derivatives to the scalar function $p$ on the normal coordinate is commutative, namely,
$\partial_{\bar r}\partial_{\bar \theta}p(\bar \theta, \bar r,\bar z)-\partial_{\bar \theta}\partial_{\bar r}p(\bar \theta,\bar r,\bar z)=0$.
This fundamental observation is the key to extract the local property of the pressure. 

\begin{remark}
It should be noticed that Enciso and Peralta-Salas \cite{EP}
considered the existence of Beltrami fields $u$
with a nonconstant proportionality factor $f$:
\begin{equation}\label{Beltrami field}
\nabla \times u=fu,\quad \nabla\cdot u=0\quad\text{in}\quad \mathbb{R}^3.
\end{equation} 
It is well known that a Beltrami field is also a solution of the steady Euler equation in $\mathbb{R}^3$.
They showed that for a generic function $f$, the only vector field $u$ satisfying
\eqref{Beltrami field} is the trivial one $u\equiv 0$.
See (2.12), (3.4) and  (3.6) in \cite{EP} for the specific condition on $f$.
Note that $g_{ij}$  (induced metric of the level set of $f$) is the fundamental component of the condition.
It would be also interesting to consider whether we can apply their method to our unsteady flow problem, and compare with our method. 
\end{remark}

For any point $x\in\mathbb{R}^3$ near the arc-length trajectory $\phi$ is uniquely  expressed as $x=\phi(\bar \theta)+\bar r n(\bar\theta)+\bar z b(\bar\theta)$ with $(\bar\theta,\bar r,\bar z)\in\mathbb{R}^3$ (the meaning of the parameters $s$ and $\bar \theta$ are the same along the arc-length trajectory).
By the Frenet-Serret formulas, we have that 
\begin{eqnarray*}
\partial_{\bar \theta}x&=&
\tau+\bar r(Tb-\kappa \tau)+\bar z \kappa n,\\
\partial_{\bar r}x&=& n,\\
\partial_{\bar z}x&=&b.
\end{eqnarray*}
This means that 
\begin{equation*}
\begin{pmatrix}
\partial_{\bar \theta}\\
\partial_{\bar  r}\\
\partial_{\bar z}
\end{pmatrix}
=
\begin{pmatrix}
1-\kappa \bar r& \bar z\kappa& \bar r T\\
0& 1& 0\\
0& 0& 1
\end{pmatrix}
\begin{pmatrix}
\tau\\
n\\
b\\
\end{pmatrix}.
\end{equation*}
\begin{remark}
For any smooth scalar function $f$,
we have 
\begin{equation*}
\partial_{\bar \theta}f(x)=\nabla f\cdot \partial_{\bar \theta}x.
\end{equation*}
$\nabla f$ itself is essentially independent of any coordinates, thus we can regard a partial derivative as the corresponding vector.
\end{remark}
Then we have  the following inverse matrix:
\begin{equation*}
\begin{pmatrix}
\tau\\
n\\
b\\
\end{pmatrix}
=
\begin{pmatrix}
(1-\kappa \bar r)^{-1}& -\bar zT (1-\kappa \bar r)^{-1}& -\bar r T(1-\kappa \bar r)^{-1}\\
0& 1& 0\\
0& 0& 1
\end{pmatrix}
\begin{pmatrix}
\partial_{\bar \theta}\\
\partial_{\bar  r}\\
\partial_{\bar z}
\end{pmatrix}.
\end{equation*}
Therefore we have the following orthonormal moving frame:
 $\partial_{\bar r}=n$, $\partial_{\bar z}=b$ and 
$$
(1-\kappa \bar r)^{-1}\partial_{\bar \theta}-\bar z T(1-\kappa \bar r)^{-1}\partial_{\bar r}-\bar rT(1-\kappa \bar r)^{-1}\partial_{\bar z}=\partial_\tau.
$$

\begin{lem}\label{Euler flow along the trajectory}
We see  $-\nabla p\cdot \tau
=D_t|u|:=\partial_{t}|u(\Phi_*(x,t),t)|$ along the trajectory.
\end{lem}
\begin{proof}
Let us define a unit tangent vector $\tilde\tau$ (in time $t$) as follows:
\begin{equation*}
\tilde \tau(t):=\frac{u(\Phi_*(x,t),t)}{|u(\Phi_*(x,t),t)|}.
\end{equation*}
Note that  there is a re-parametrize factor $s(t)$ such that 
\begin{equation*}
\tau(s(t))=\tilde \tau(t).
\end{equation*}
Since $u\cdot \partial_s\tau=0$, we see that 
\begin{eqnarray*}
\partial_{t}|u(\Phi_*(x,t),t)|&=&
\partial_{t}(u(\Phi_*(x,t),t)\cdot \tilde\tau(t))\\
&=&
\partial_{t}(u(\Phi_*(x,t),t))\cdot \tilde\tau(t)+
u(\Phi_* (x,t),t)\cdot\partial_s\tau\partial_{t}s\\
&=&\partial_{t}(u(\Phi_*(x,t),t))\cdot \tilde\tau(t).
\end{eqnarray*}
By the above calculation we have 
\begin{equation*}
-\nabla p\cdot\tau
=\partial_{t}(u(\Phi_*(x,t),t)\cdot \tau=\partial_{t}(u(\Phi_*(x,t),t)\cdot\tilde\tau=
D_{t}|u|.
\end{equation*}
\end{proof}

We now rewrite the Euler equations by using curvature and torsion.
\begin{lem}\label{formulas on curvature and torsion}
Along the arc-length trajectory, we have 
\begin{eqnarray*}
3\kappa D_t|u|+\partial_s\kappa|u|^2
=
\partial_{\bar r}D_t|u|
\end{eqnarray*}
and 
\begin{equation*}
T\kappa|u|^2
=\partial_{\bar z}D_t|u|.
\end{equation*}
\end{lem}

\begin{proof}
Let us re-define $\phi(s)=\Phi_*(x,t(s))$ with  smooth function $t(s)$ satisfying  $\partial_st=|u|^{-1}$.
We see that 
\begin{equation*}
\partial_s\phi\cdot \tau=1.
\end{equation*}
By the unit normal vector with the curvature constant, we see
\begin{equation*}
\kappa n=\partial_s^2\phi=\partial_s(\partial_t\Phi_*\partial_st)
=\partial_t^2\Phi_*(\partial_st)^2+\partial_t\Phi_*\partial_s^2t.
\end{equation*}
Recall the Euler equation: $\partial_t^2\Phi_*=-\nabla p$.
Then we have 
\begin{eqnarray*}
-(\nabla p\cdot n)
&=&
(\partial_t^2\Phi_*\cdot n)=
\kappa|u|^2,\\
-\partial_s(\nabla p\cdot n)
&=&\partial_s(\kappa(\partial_st)^{-2})
=\partial_s\kappa(\partial_s t)^{-2}-2\kappa(\partial_st)^{-3}(\partial_s^2t),\\
-\nabla p\cdot \tau
&=&-|u|^3\partial_s^2 t,\\
-\nabla p\cdot b
&=&0.
\end{eqnarray*}
Note that $\partial_s^2 t$ is unknown, so we now figure out it by  Lemma \ref{Euler flow along the trajectory} and the above third equality:
\begin{equation*}
\partial_s^2t=-|u|^{-3}\partial_t|u|.
\end{equation*}
Along the arc-length trajectory, we have (recall $\partial_{\bar\theta}=\partial_s$)
\begin{eqnarray*}
-\partial_{\bar r}(\nabla p\cdot \tau)
&=&-\partial_{\bar r}\partial_\tau p
\\
&=&
-\kappa\partial_{\bar \theta} p-\partial_{\bar r}\partial_{\bar \theta} p-T\partial_{\bar z} p\\
\nonumber
(\text{commute}\ \partial_{\bar r}\ \text{and}\  \partial_{\bar \theta})
&=&
-\kappa(\nabla p\cdot \tau)-\partial_{\bar\theta} (\nabla p\cdot n)-T(\nabla p\cdot b)\\
&=&
-\kappa |u|^3\partial_s^2t+\partial_s\kappa(\partial_st)^{-2}-2\kappa(\partial_st)^{-3}(\partial_s^2 t)
\\
&=&
3\kappa\partial_t|u|+\partial_s\kappa|u|^2.
\end{eqnarray*}
Since $\nabla p\cdot b=\partial_{\bar z} p\equiv 0$ along the trajectory, then
\begin{eqnarray*}
-\partial_{\bar z}(\nabla p\cdot \tau)|_{\bar r,\bar z=0}
&=&
-\partial_{\bar z}\partial_\tau p|_{\bar r,\bar z=0}=
-\partial_{\bar z}\partial_{\bar\theta} p-T\partial_{\bar r} p
=
-T(\nabla p\cdot n)\\
&=&
T\kappa|u|^2.
\end{eqnarray*}

By Lemma \ref{Euler flow along the trajectory}
along the arc-length trajectory $\phi$, we have 
\begin{equation*}
3\kappa\partial_t|u|+\partial_s\kappa|u|^2
=
-\partial_{\bar r}(\nabla p\cdot \tau)|_{\bar r,\bar z=0}
=
\partial_{\bar r}D_t|u|
\end{equation*}
and 
\begin{equation*}
T\kappa|u|^2
=
-\partial_{\bar z}(\nabla p\cdot \tau)|_{\bar r,\bar z=0}
=
\partial_{\bar z}D_t|u|
.\\
\end{equation*}
\end{proof}

\section{Proof of the main theorem (the pulsatile flow case).}

To prove the main theorem, it is enough to show the following lemma:
\begin{lem}
Let $t_j>0$ ($j=1,2,\cdots$) be fixed.
For any $x\in \Phi(D_\gamma, t_j)$, there is $\beta>0$ such that
 $\beta \lesssim u(x,t_j)\cdot e_\theta\lesssim \beta^{-1}$, $x\cdot e_r>2\beta$
and  $|L^0(\tilde r_0(t_j),z(t_j),t_j)|+|L^x(\tilde r_0(t_j),z(t_j),t_j)|\leq 1/(2\beta)$.
For any $\epsilon>0$, 
 then there is  $\delta>0$
such that 
for any small time interval $I$ with initial time $t_j$,  
at least either of the following four cases must happen:
\begin{itemize}

\item
$L^x(\tilde r_0(t),z(t),t), L^0(\tilde r_0(t),z(t),t)>1/\beta$,

\item  
$L^t(\tilde r_0(t),z(t),t)\gtrsim
1/\epsilon$, 



\item  
  $|\Phi_*(x,t)\cdot e_r|<\beta$,

\item

$\tilde r_0(t)<\beta$,

\end{itemize}
for some $t\in I$, with any inflow $g(t)$ satisfying
\begin{equation*}
 g(t)
\approx 1,
\quad
|g'(t)|<1/\epsilon
\quad\text{and}\quad
1/\delta\approx |g''(t)|
\quad\text{in}\quad
t\in I,
\end{equation*}
where $\tilde r_0(t)$ and $z(t)$ are determined by $\tilde\Phi(\tilde r_0(t), z(t),t)=\Phi_*(x,t)$
(in this case $\Phi_*(x,t_j)=x$).
Since $\Phi(D_\gamma,t)$ is always compact and the solution is always smooth,
$\delta$ can be independent of the choice of $x\in\Phi(D_\gamma,t)$.
\end{lem}

Since the time interval $I$ is arbitrary, we see that $L^0$ or $L^x$ or $\Phi_*\cdot e_r$ or $\tilde r_0$ is not continuous at $t_j$,
or $L^t\gtrsim 1/\epsilon$ 
 for some $t\in I$.
The discontinuity contradicts the smoothness property, 
thus 
\begin{equation*}
L^t\gtrsim 
1/\epsilon
\end{equation*}
only occurs.


\begin{proof}
In what follows, we prove the above lemma.
For any small time interval $I$,
assume that the axisymmetric smooth Euler flow satisfies the following  conditions:
\begin{itemize}
\item
$L^x(\tilde r_0(t),z(t),t), L^0(\tilde r_0(t),z(t),t)\leq 1/\beta$ 
and
$
L^t(\tilde r_0(t),z(t),t)
\lesssim 1/\epsilon$ 



\item

$|\Phi_*(x,t)\cdot e_r|\geq \beta$ 
and 
$\tilde r_0(t)\geq \beta$
\end{itemize}
for any $t\in I$, where 
 $(\tilde r_0(t),z(t))=(\tilde \Phi^{-1}\circ\Phi_*)(x,t)$, and we employ a contradiction argument.
By the second assumption: $|\Phi_*(x,t)\cdot e_r|\geq \beta$, $R$ satisfies the following: 
\begin{equation*}
R(Z_*(t))=R_*(t)\geq\beta
\quad\text{for}\quad t\in I.
\end{equation*}

By the explicit formulas in Section \ref{explicit formulas}, we have the following lemma (these are direct calculations, thus we omit its proof).
\begin{lem}\label{estimates of v}
For $t=Z^{-1}_{*t}\in I$,
we have the following estimates along the axis-length trajectory:
\begin{equation}\label{partial_zv_z}
\begin{cases}
|\partial_zv_z(R(z),z,Z^{-1}_{*t}(z))|\lesssim 1/\epsilon,\\
|\partial_z^2v_z(R(z),z,Z^{-1}_{*t}(z))|\approx 1/\delta,\\
|\partial_zv_r(R(z),z,Z^{-1}_{*t}(z))|\lesssim  1/\epsilon,\\
|\partial_z^2v_r(R(z),z,Z^{-1}_{*t}(z))|\lesssim  1/\delta.
\end{cases}
\end{equation}
Moreover, we have
\begin{eqnarray}\label{v_theta estimate}
& &\beta\lesssim |v_\theta(R_*(z),z,Z^{-1}_{*t}(z))|\lesssim 1/\beta,\\ 
\label{partial_zv_theta}
& &|\partial_zv_\theta(R(z),z,Z^{-1}_{*t}(z))|\lesssim 1/\beta,\\
\nonumber
& &|\partial_z^2v_\theta(R(z),z,Z^{-1}_{*t}(z))|\lesssim 1/\epsilon,
\end{eqnarray}
\begin{equation}\label{partial u_t}
\partial_t|u(\Phi(x,t),t)|\lesssim 1/\epsilon,
\end{equation}
\begin{equation}\label{partial u_t}
\partial_z\partial_t|u(\Phi(x,t),t)|,\ \partial_r\partial_t|u(\Phi(x,t),t)|\lesssim 1/\epsilon.
\end{equation}
\end{lem}

By the above lemma with Remark \ref{remark on explicit formulas}, we immediately have  $|\Theta''|\lesssim 1/\epsilon$ and $\Theta'''\approx 1/\delta$ (for sufficiently small $\delta$ compare with $\epsilon$) in $t\in I$.

\begin{lem}\label{key estimate on torsion}
For any $\epsilon>0$,
we have 
\begin{eqnarray*}
|u|^2|\partial_s\kappa|\gg \kappa D_t|u|
\end{eqnarray*}
for sufficiently small $\delta>0$.
\end{lem}
\begin{proof}
From Section \ref{estimates on curvature and torsion}, we see
\begin{eqnarray*}
\partial_s(\kappa^2)
&=&
 2(\partial_s\kappa) \kappa=
2R\Theta''(R\Theta''')(1+(R')^2)^{-5/2}+\text{remainder},\\
\kappa&=& |R\Theta''|(1+(R')^2)^{-1}+\text{remainder},\\
\partial_s\kappa&=&\frac{R\Theta''(R\Theta''')(1+(R')^2)^{-5/2}}{\kappa}+\text{remainder}\\
&=&
-R\Theta'''(1+(R')^2)^{-3/2}+\text{remainder}\\
&\approx&
1/\delta.
\end{eqnarray*}
in $t\in I$.
``remainder" is small compare with the main terms provided by small $\epsilon, \delta>0$.
Thus we immediately obtain
$|u|^2|\partial_s\kappa|\gg \kappa D_t|u|$
for sufficiently small $\delta>0$.
\end{proof}

By Lemma \ref{formulas on curvature and torsion}, we see
\begin{equation*}
0\geq \bigg|\partial_s\kappa|u|^2\bigg|-\bigg|\partial_rD_t|u|\bigg|
-\bigg|\partial_zD_t|u|\bigg|-\bigg|3\kappa D_t|u|\bigg|
\end{equation*}
and it is in contradiction, since $\partial_s\kappa$ is sufficiently large compare with the other terms.

\end{proof}

\section{Proof of the main theorem (the vortex breakdown case)}

Assume
\begin{equation*}
|\partial_tv_z(0,z,0)|+
|\partial_t\partial_rv_r(0,z,0)|+|\partial_t\partial_rv_z(0,z,0)|+
|\partial_t\partial_zv_z(0,z,0)|+|\partial_t\partial_r^2v_z(0,z,0)|
\leq 1/\epsilon_2
\end{equation*}
and employ a contradiction argument.
Recall that, by Remark \ref{remark on explicit formulas}, 
$|\Theta'|\approx1/\epsilon_1$, $|\Theta''|\lesssim 1/\delta$ and $|\Theta'''|\lesssim1/\delta$ in some small time interval.
From Section \ref{estimates on curvature and torsion}, near the axis, we have ($r=R(r_0,z_0,t)$)
\begin{equation*}
\kappa= r(\Theta')^2+O(r^2)
\quad\text{and}\quad
\partial_s\kappa= C(\delta)r+O(r^2).
\end{equation*}
Thus
near the axis, we have
\begin{equation*}
3\kappa D_t|u|+\partial_s\kappa |u|^2
\approx (\Theta')^2r+O(r^2).
\end{equation*}
Since $\partial_{\bar r}D_t|u|=3\kappa D_t|u|+\partial_s\kappa|u|^2$ and $\kappa=\partial_s\kappa=0$ along the axis, we have $\partial_{\bar r}D_t|u|=0$ along the axis.
By the mean value theorem,  we have 
\begin{equation*}
 \partial_{r}^2D_t|u|\approx (\Theta')^2.
\end{equation*}
along the axis (note that $\partial_{\bar r}\to\partial_r$ if the corresponding point approaches  the axis).
However it is in contradiction, since the right hand side is large, while the left hand side is not large.

\section{Appendix: Instantaneous blow-up}

In this section we show instantaneous blow-up.
Let us consider the Euler equations in the whole space $\mathbb{R}^3$:
\begin{eqnarray}
\label{Euler eq. in R^3}
& &\partial_tu+(u\cdot \nabla)u=-\nabla p,\quad\nabla \cdot u=0\quad \text{in}
\quad \mathbb{R}^3,\\
\nonumber
& & \quad u|_{t=0}=u_0.
\end{eqnarray}

The first existence results for \eqref{Euler eq. in R^3} were proved in the framework of H\"older spaces by 
Gyunter \cite{Gu}, Lichtenstein \cite{Li} and Wolibner \cite{Wo}. 
More refined results 
were obtained subsequently by 
Kato \cite{Ka}, Swann \cite{Sw}, Bardos and Frisch \cite{BF}, Ebin \cite{Eb}, Chemin \cite{Ch}, 
Constantin \cite{Co1} and Majda and Bertozzi \cite{MB} among others. 
On the other hand, Bardos and Titi \cite{BT} found examples of solutions in H\"older spaces $C^\alpha$ 
and the Zygmund space $B^1_{\infty,\infty}$ which exhibit an instantaneous loss of smoothness 
in the spatial variable for any $0<\alpha<1$ (see also \cite{CS,MY}). 
Similar examples in logarithmic Lipschitz spaces $\mathrm{logLip}^\alpha$ were given by 
the authors in \cite{MY}. 
In another direction Cheskidov and Shvydkoy \cite{CS} constructed periodic solutions that are discontinuous 
in time at $t=0$ in the Besov spaces $B^s_{p, \infty}$ where $s>0$ and $2<p\leq\infty$. 
After their work, in a series of papers Bourgain and Li \cite{BL, BL1} constructed smooth solutions 
which exhibit 
instantaneous blowup 
in borderline spaces such as $W^{n/p+1,p}$ for any $1 \leq p < \infty$ and 
$B^{n/p+1}_{p,q}$ for any $1 \leq p < \infty$ and $1 < q \leq \infty$ 
as well as in the standard spaces $C^k$ and $C^{k-1,1}$ for any integer $k \geq 1$; 
see also Elgindi and Masmoudi \cite{ElMa} and \cite{MY2}. 
As observed in \cite{BL1} the cases $C^k$ and $C^{k-1,1}$ are particularly intriguing in view of 
the classical existence and uniqueness results mentioned above. 

In \cite{MY3} (see also \cite{MY2}), they revisited the picture of local well-posedness in the sense of Hadamard 
for the Euler equations in H\"older spaces. 
They present a simple example based on a DiPerna-Majda type shear flow which shows that in general 
the data-to-solution map of \eqref{Euler eq.} is not continuous into the space $L^\infty([0,T),C^{1,\alpha})$ 
for any $0<\alpha<1$. 
On the other hand, continuity of this map is restored (in the strong sense) if the Cauchy problem 
is restricted to the so called little H\"older space $c^{1,\alpha}$.

\begin{remark}\label{wellposedness}
For $u_0\in c^{2,\alpha}$, 
we can also show that there exists a unique solution $u$ which is in (see \cite[Section 4.4]{MB} and \cite{MY3} for example)
\begin{equation*}
C([0,T]:c^{2,\alpha}(\mathbb{R}^3))\cap C^1([0,T]:c^{1,\alpha}(\mathbb{R}^3))\cap C^2([0,T]:c^{0,\alpha}(\mathbb{R}^3)).
\end{equation*}
Therefore, if the solution $u$ is axi-symmetric, then  the corresponding components $v_r$ and $v_z$ satisfy
\begin{equation*}
|\partial_tv_r(0,z_j,t)|+|\partial_tv_z(0,z_j,t)|\lesssim 1,
\end{equation*} 
\begin{equation*}
|\partial_t\partial_rv_r(0,z_j,t)|+|\partial_t\partial_rv_z(0,z_j,t)|+
|\partial_t\partial_zv_z(0,z_j,t)|\lesssim 1
\end{equation*}
and
\begin{equation*}
|\partial_t^2v_r(0,z_j,t)|+|\partial_t^2v_z(0,z_j,t)|\lesssim 1\quad\text{for}\quad t\in[0,T].
\end{equation*}
\end{remark}
In this appendix, we  show that even if the solution to the Euler equations is wellposed, such as, in $c^{2,\alpha}$, 
it may blows up (in some norm) instantaneously.

\begin{theorem}
There is an axisymmetric initial data  $u_0\in c^{2,\alpha}(\mathbb{R}^3)$ such that the corresponding unique solution $u$ 
is not in $C^1([0,T]:C^2(\mathbb{R}^3))$ for any $T>0$.
More precisely, we choose an axisymmetric initial data as the following: there is sufficiently small $\beta>0$ such that for any $\{\epsilon_j\}_j$ $(\epsilon_j\to 0)$ and  $\{z_j\}_j$ $(z_j\to z)$, there is 
$\{\delta_j\}_j$ ($\delta_j\to 0$ as  $j\to\infty$)
such that 
\begin{eqnarray*}
|v_z(0,z_j,0)|&\approx& 1,\\
|\partial_z\partial_rv_\theta(0,z_j,0)|&\approx& 1/\beta,\\
|\partial_z\partial_r^2v_\theta(0,z_j,0)|&\approx& 1/\delta_j,\\
\sum_{\stackrel{0\leq j+k\leq 2}{(j,k)\not=(2,1)}}|\partial_r^j\partial_z^kv_\theta(0,z_j,0)|&\lesssim& 1,\\
\sum_{0\leq j+k\leq 3}|\partial_r^j\partial_z^kv_r(0,z_j,0)|&\lesssim& 1,\\
\sum_{1\leq j+k\leq 3}|\partial_r^j\partial_z^kv_z(0,z_j,0)|&\lesssim& 1. 
\end{eqnarray*}
Then we have 
\begin{equation*}
|\partial_t^2\partial_rv_r(0,z_j,0)|+|\partial_t^2\partial_rv_z(0,z_j,0)|+
|\partial_t^2\partial_zv_r(0,z_j,0)|+|\partial_t^2\partial_zv_z(0,z_j,0)|
>1/\epsilon_j.
\end{equation*}
\end{theorem}

\begin{proof}

The proof is similar to  the  ``vortex breakdown" case.
By Remark \ref{remark on explicit formulas},  we can figure out that 
$\Theta'''|_{r=0}$ is large for some small time interval.
The same argument holds true that 
$\Theta'|_{r=0}$ and $\Theta''|_{r=0}$ are not large. Due to Remark \ref{wellposedness}, we see $R'$, $R''$ and $R'''$ are all small.
Let $r=R(r_0,z_0,t)$. 
By Lemma \ref{Euler flow along the trajectory}, near the axis, we see 
\begin{equation*}
\kappa= r\Theta''+O(r^2),\quad\partial_s\kappa=\Theta'''r+O(r^2).
\end{equation*}
Thus
near the axis, we have
\begin{equation*}
3\kappa D_t|u|+\partial_s\kappa |u|^2
\approx \Theta'''r+O(r^2).
\end{equation*}
By the same argument as in the previous section, we have 
\begin{equation*}
 \partial_{r}^2D_t|u|\approx \Theta'''.
\end{equation*}
along the axis.
This estimate tells us that $|\partial_t\partial_r^2v_r(0,z_j,0)|\approx 1/\delta_j$.

\end{proof}

\vspace{0.5cm}
\noindent
{\bf Acknowledgments.}\ 
The author would like to thank Professor Norikazu Saito for letting me know 
the book \cite{FQV}, Professor Hiroshi Suito for letting me know 
``Womersley number", and also
Doctor Kento Yamada for letting me know the articles \cite{BVS,E,H,S}.
The author  was partially supported by 
Grant-in-Aid for Young Scientists A (17H04825),
Japan Society for the Promotion of Science (JSPS),
and also partially supported by
 JST CREST.

\bibliographystyle{amsplain}

\begin{thebibliography}{10} 


\bibitem{BF} 
C. Bardos and U. Frisch, 
\textit{Finite-time regularity for bounded and unbounded ideal incompressible fluids using H\"older estimates}, 
Turbulence and Navier-Stokes equations (Proc. Conf., Univ. Paris-Sud, Orsay, 1975), 
Lecture Notes in Math., vol. \textbf{565}, Springer, Berlin 1976

\bibitem{BT} 
C. Bardos and E. Titi, 
\textit{Loss of smoothness and energy conserving rough weak solutions for the 3d Euler equations}, 
Discrete Cont. Dyn. Syst. ser. \textbf{S3} (2010), 185-197. 



\bibitem{BL} 
J. Bourgain and D. Li, 
\textit{Strong ill-posedness of the incompressible Euler equations in borderline Sobolev spaces}, 
Invent. math. \textbf{201}, (2015), 97-157; 
preprint arXiv:1307.7090 [math.AP]. 

\bibitem{BL1} 
J. Bourgain and D. Li, 
\textit{Strong illposedness of the incompressible Euler equation in integer $C^m$ spaces}, 
Geom. funct. anal. \textbf{25} (2015), 1-86; 
preprint arXiv:1405.2847 [math.AP]. 

\bibitem{BVS}
M. Brons, L.K. Voigt and J. N. Sorensen,
\textit{Streamline topology of steady axisymmetric vortex breakdown in a cylinder with co- and counter-rotating end-covers},
J. Fluid Mech. \textbf{401}, (1999), 275-292.  

\bibitem{C0}
D. Chae,
\textit{On the Lagrangian dynamics for the 3D incompressible Euler equations},
Comm. Math. Phys., \textbf{269}, (2007), 557-569. 

\bibitem{C1}
D. Chae, 
\textit{On the blow-up problem for the axisymmetric 3D Euler equations},
Nonlinearity, \textbf{21}, (2008), 2053-2060.

\bibitem{C2}
D. Chae,
\textit{On the Lagrangian dynamics of the axisymmetric 3D Euler equations},
J. Diff. Eq., \textbf{249} (2010), 571-577.













\bibitem{CCY}
C-H. Chan, M. Czubak and T. Yoneda, 
\emph{An ODE for boundary layer separation on a sphere and a hyperbolic space,} 
 Physica D, \textbf{282} (2014), 34-38.

\bibitem{Ch} 
J. Chemin, 
\textit{Perfect Incompressible Fluids}, 
Clarendon Press, Oxford 1998. 



\bibitem{CS} 
A. Cheskidov and R. Shvydkoy, 
\textit{Ill-posedness of basic equations of fluid dynamics in Besov spaces}, 
Proc. A.M.S. \textbf{138} (2010), 1059-1067. 

\bibitem{Co1} 
P. Constantin, 
\textit{An Eulerian-Lagrangian approach for incompressible fluids: local theory}, 
J. Amer. Math. Soc. \textbf{14} (2001), 263-278. 








\bibitem{Eb} 
D. Ebin, 
\textit{A concise presentation of the Euler equations of hydrodynamics}, 
Comm. Partial Differential Equations \textbf{9} (1984), 539-559. 

\bibitem{ElMa} 
T. Elgindi and N. Masmoudi, 
\textit{$L^\infty$ ill-posedness for a class of equations arising in hydrodynamics}, 
preprint arXiv:1405.2478 [math.AP]. 



\bibitem{EP}
A. Enciso and D. Peralta-Salas,
\textit{Beltrami fields with a nonconstant proportionality factor are rare},
Arch. Rational Mech. Anal., \textbf{220} (2016), 243-260.

\bibitem{E}
M.P. Escudier,
\textit{Observations of the flow produced in a cylindrical container by a rotating endwall},
Experiments in Fluids, \textbf{2} (1984), 189-196.




\bibitem{FQV}
L. Formaggia, A. Quarteroni and A. Veneziani,
\textit{Cardiovascular mathematics, modeling and simulation of the circulatory system},
Springer-Verlag, Italia, Milano, 2009. 







\bibitem{Gu} 
N. Gyunter, 
\textit{On the motion of a fluid contained in a given moving vessel}, 
(Russian), Izvestia Akad. Nauk USSR, Ser. Phys. Math. \textbf{20} (1926), 1323-1348, 1503-1532; 
\textbf{21} (1927), 621-556, 735-756, 1139-1162; \textbf{22} (1928), 9-30. 


\bibitem{H}
M.G. Hall,
\emph{Vortex breakdown},
Annu. Rev. Fluid Mech., \textbf{4} (1972), 195-218.











\bibitem{HNY}
P-Y. Hsu, H. Notsu, T. Yoneda, 
\emph{
A local analysis of the axi-symmetric Navier-Stokes flow near a saddle point and no-slip flat boundary,} J. Fluid Mech., 794 (2016) 444-459. 

\bibitem{HRT}
J. G. Heywood, R. Rannacher and S. Turek,
\emph{
Artificial boundaries and flux and pressure conditions for the incompressible Navier-Stokes equations
},
Int. J. Numerical Methods in Fluid, \textbf{22} (1996), 325-352.


\bibitem{Ka} 
T. Kato, 
\textit{On classical solutions of the two-dimensional non-stationary Euler equation}, 
Arch. Ration. Mech. Anal. \textbf{25} (1967), 188-200. 







\bibitem{Li} 
L. Lichtenstein, 
\textit{Uber einige Existenzprobleme der Hydrodynamik}, 
Math. Zeit. \textbf{23} (1925), 89-154, 309-316; \textbf{26} (1927), 196-323; \textbf{28} (1928), 387-415; 
\textbf{32} (1930), 608-640. 


\bibitem{MW}
T. Ma and S. Wang,
\emph{
Boundary layer separation and structural bifurcation for 2-D incompressible fluid flows.
 Partial differential equations and applications,} Discrete Contin. Dyn. Syst., 10 (2004),  459--472.

\bibitem{MM}
Y. Maekawa and A. Mazzucato,
\emph{Inviscid limit and boundary layers for Navier-Stokes flows,}
to appear in Handbook of Mathematical Analysis in Mechanics of Viscous Fluids,
Y. Giga and A. Novotn\'y Ed., Springer;  arXiv:1610.05372.

\bibitem{MB} 
A. Majda and A. Bertozzi, 
\textit{Vorticity and Incompressible Flow}, 
Cambridge University Press, Cambridge 2002. 


\bibitem{MY} 
G. Misio{\l}ek and T. Yoneda, 
\textit{Ill-posedness examples for the quasi-geostrophic and the Euler equations}, 
Analysis, geometry and quantum field theory, Contemp. Math. \textbf{584}, 
Amer. Math. Soc., Providence, RI, 2012, 251-258. 

\bibitem{MY2} 
G. Misio{\l}ek and T. Yoneda, 
\textit{Local ill-posedness of the incompressible Euler equations in $C^1$ and $B^1_{\infty,1}$}, 
Math. Ann. \textbf{364} (2016), 243-268; 
Erratum, \textbf{363} (2015), 1399-1400. 


\bibitem{MY3}
G. Misio{\l}ek and T. Yoneda,
\textit{Continuity of the solution map of the Euler equations in H\"older spaces and weak norm inflation in Besov spaces}, to appear in Trans. Amer. Math. Soc. 




\bibitem{S}
T. Sarpkaya,
\textit{On stationary and travelling vortex breakdowns},
J. Fluid Mech., \textbf{45} (1971), 545-559.


\bibitem{Sw} 
H. Swann, 
\textit{The existence and uniqueness of nonstationary ideal incompressible flow in bounded domains in $R_3$}, 
Trans. Amer. Math. Soc. \textbf{179} (1973), 167-180. 


\bibitem{TKWP}
R. Trip, D.J. Kuik, J. Westerweel and C. Poelma,
\emph{
An experimental study of transitional pulsatile pipe flow,
}
Phys. Fluids, 24 (2012), 014103.


\bibitem{Wo} 
W. Wolibner, 
\textit{Un theor\'eme sur l'existence du mouvement plan d'un fluide parfait, homog\'ene, 
incompressible, pendant un temps infiniment long}, 
Math. Z. \textbf{37} (1933), 698-726. 



\bibitem{W}
J. R. Womersley,
\emph{Method for the calculation of velocity, rate of flow and viscous drag
in arteries when the pressure  gradient is known,
}
J. Physiol., 127 (1955), 553--563.































\end{thebibliography}

\end{document}